\documentclass[reqno,11pt]{amsart}

\usepackage{amsmath}
\usepackage{amsthm}
\usepackage{amssymb}
\usepackage{amscd}
\usepackage{xypic}
\usepackage{verbatim}
\usepackage{stmaryrd}
\usepackage[plainpages=false,colorlinks,hyperindex,pdfpagemode=None,bookmarksopen,linkcolor=red,citecolor=blue,urlcolor=blue]{hyperref}
\usepackage{pdflscape}
\usepackage{palatino}

\usepackage{etoolbox}
\makeatletter
\patchcmd{\@maketitle}
  {\ifx\@empty\@dedicatory}
  {\ifx\@empty\@date \else {\vskip3ex \centering\footnotesize\@date\par\vskip1ex}\fi
   \ifx\@empty\@dedicatory}
  {}{}
\patchcmd{\@adminfootnotes}
  {\ifx\@empty\@date\else \@footnotetext{\@setdate}\fi}
  {}{}{}
\makeatother

\swapnumbers
\theoremstyle{plain}
\newtheorem{theorem}[subsection]{Theorem}

\newtheorem*{theorem*}{Theorem}
\newtheorem{proposition}[subsection]{Proposition}
\newtheorem*{proposition*}{Proposition}
\newtheorem{lemma}[subsection]{Lemma}
\newtheorem*{lemma*}{Lemma}
\newtheorem*{fact*}{Fact}
\newtheorem{corollary}[subsection]{Corollary}
\newtheorem*{corollary*}{Corollary}

\theoremstyle{definition}
\newtheorem{definition}[subsection]{Definition}

\theoremstyle{remark}
\newtheorem{remark}[subsection]{Remark}

\newtheorem{example}[subsection]{Example}

\renewcommand{\comment}[1] {  }

\DeclareFontFamily{OT1}{rsfs}{}
\DeclareFontShape{OT1}{rsfs}{n}{it}{<-> rsfs10}{}
\DeclareMathAlphabet{\mathscr}{OT1}{rsfs}{n}{it}

\newcommand{\CC}{\mathbb{C}}

\newcommand{\QQ}{\mathbb{Q}}

\newcommand{\Hom}{\operatorname{Hom}}

\newcommand{\Gm}{\mathbb{G}_m}

\newcommand{\GL}{\operatorname{GL}}
\newcommand{\Mat}{\operatorname{Mat}}

\newcommand{\PGL}{\operatorname{PGL}}

\newcommand{\SL}{\operatorname{SL}}
\newcommand{\Sp}{\operatorname{Sp}}

\newcommand{\diag}{{\operatorname{diag}}}

\renewcommand{\span}{\operatorname{span}}

\newcommand{\std}{{\operatorname{std}}}

\newcommand{\Asymp}{\operatorname{Asymp}}
\newcommand{\Sat}{\operatorname{Sat}}

\begin{document}
\numberwithin{equation}{section}
\setcounter{tocdepth}{1}
\title{Inverse Satake transforms}
\author{Yiannis Sakellaridis}
\address{Rutgers University - Newark \\
101 Warren Street \\
Smith Hall 216 \\
Newark, NJ 07102\\
USA.}
\email{sakellar@rutgers.edu}

\begin{abstract}
Let $H$ be a split reductive group over a local non-archimedean field, and let $\hat H$ denote its Langlands dual group. 
We present an explicit formula for the generating function of an unramified $L$-function associated to a highest weight representation of the dual group, considered as a series of elements in the Hecke algebra of $H$. This offers an alternative approach to a solution of the same problem by Wen-Wei Li. Moreover, we generalize the notion of ``Satake transform'' and perform the analogous calculation for a large class of spherical varieties.
\end{abstract}

\maketitle

\tableofcontents

\section{Introduction}

The purpose of this paper is to discuss the notion of ``Satake transform'' for a spherical variety $X$ under a reductive group $G$ over a $p$-adic field $F$, generalizing the corresponding notion for the case $X=$ a reductive group, and to present an explicit formula for its inversion, based on the theory of asymptotics of smooth $G$-representations on $X$. In particular, in the group case this gives an alternative approach to the one discovered by Wen-Wei Li in \cite{WWL}.

Let $H$ be a reductive group over a $p$-adic field $F$, and let us assume for simplicity that $H$ is split. The Satake transform establishes an isomorphism between the unramified Hecke algebra of $H$ (with respect to some hyperspecial open compact subgroup) and the algebra of invariant polynomials on the complex dual group $\check H$. 

While the inversion of this transform is known to be given by Kazhdan-Lusztig polynomials, one would like an ``efficient'' method of inversion, especially when invariant polynomials on $\check H$ are replaced by certain rational functions on $\check H$ (which, on the other side of the Satake isomorphism, should correspond to series of elements in the Hecke algebra). In particular, one is motivated by the paper of B.C.\ Ng\^o \cite{Ngo-monoids}, which suggests a relationship between ``basic functions'' on reductive monoids and local unramified automorphic $L$-functions. The prototype of this is the characteristic function of $n\times n$ integers, which was famously used by Godement and Jacquet \cite{GJ} to study the standard $L$-function for $\GL_n$. Ng\^o's discovery shows that, at least in the local, unramified setting, the Godement-Jacquet approach is part of a more general story, where \emph{all} automorphic $L$-functions appear. In order, however, to study global problems, one would need to generalize Fourier transforms and the Poisson summation formula. Part of the motivation of studying the inverse Satake transform has to do with the hope of finding an explicit (non-spectral) description of a Fourier transform in this setting. 

In the paper \cite{WWL}, Wen-Wei Li showed that the inversion problem for elements of the Hecke algebra corresponding to $L$-functions can be efficiently solved, circumventing the tedious, algorithmic process of decomposing symmetric powers of a representation of $\check H$ and then using an infinite number of Kazhdan-Lusztig polynomials. His Theorem 5.3.1 offers a complete resolution to this problem. The goal of the present paper is to offer an alternative approach, which automatically applies to spherical varieties as well.

Indeed, as was shown in \cite{SaRS}, the Godement-Jacquet method and its generalizations proposed by Ng\^o are themselves part of a broader framework which includes the Rankin-Selberg method. The basic object for this generalization is that of an \emph{affine spherical variety}. While we do not yet know the precise relationship with $L$-functions in the most general case, it seems to be confirmed by all known examples \cite{BFGM, BNS}, and it may not be long before it is established. Thus, for the same reasons as above it will be interesting to have some understanding of inverse spectral transforms in this generality and, in any case, as we will see, the theory of asymptotics on spherical varieties provides a very straightforward approach to the problem of inversion, including in the group case.

\subsection{Acknowledgments} I am very grateful to Wen-Wei Li for fruitful conversations on this topic, corrections on previous drafts, for explaining to me the relation between my formula and his, and for allowing me to reproduce this explanation here (\S \ref{sec:groupcase}). I would also like ot thank Bill Casselman who has enthusiastically encouraged us to think about this problem. This work was supported by by NSF grant DMS-1101471. The paper was completed during my stay at MSRI for the program on ``New Geometric Methods in Number Theory and Automorphic Forms'', and it is my pleasure to thank MSRI for its hospitality and support.

\section{Assumptions and Cartan decomposition} \label{sec:assumptions}

We will assume throughout that $G$ is split over $F$, in order to use results that are not yet available in the general case of unramified groups. In particular, we may fix the Chevalley model for $G$ over the ring of integers $\mathfrak o$ of $F$, and denote by $K$ its hyperspecial maximal compact subgroup $G(\mathfrak o)$. Whenever there is no confusion, we will be denoting $G(F)$ simply by $G$, $X(F)$ by $X$ etc.

We let $X$ be a homogeneous spherical variety which satisfies the following conditions:

\begin{itemize}
\item It is quasi-affine, admits a $G$-eigenmeasure, and its open Borel orbit admits a $B$-invariant measure. These assumptions cause no harm to generality, since every homogeneous variety admits a torus bundle whose total space satisfies them \cite[\S 3.8]{SaSpc}.
\item It is \emph{wavefront} \cite[\S 2.1]{SV}. This property, whose definition will be recalled below, applies to almost all spherical varieties, and certainly all symmetric ones (such as: reductive groups themselves). The reason for imposing it is to be able to take advantage of the theory of asymptotics, which is for now missing in the non-wavefront case.
\item It satisfies the conclusions of \cite[Theorems 2.3.8, 2.3.10]{SaRS} on orbits of a hyperspecial and an Iwahori subgroup. These conditions are satisfied at almost every place when $G, X$ are defined over a global field, and will also be recalled below.
\end{itemize}

To formulate the above conditions, we introduce the notion of a \emph{Cartan subtorus} $A_X\subset X$, pointing the reader to \cite[\S 2]{SV}  for more details (where the notation $A_X$ was also used, but not the term ``Cartan subtorus''). At first reading I would recommend to the reader to skip the definitions, and just accept the existence of a ``root system'' formalism allowing for a generalization of the Cartan decomposition to an arbitrary spherical variety.

A Cartan subtorus of $X$ is a subvariety obtained by fixing a triple $(x_0,B,f)$, where $x_0\in X$, $B$ is a Borel subgroup such that $x_0$ is contained in its open orbit $X^\circ$, and $f$ is a $B$-eigenfunction
 whose set-theoretic zero locus is $X\smallsetminus X^\circ$. If we pull the function $f$ back to a function on $G$ via the action map $G\ni g\mapsto x_0g\in G$, its differential becomes an element of the coadjoint representation whose centralizer we denote by $L(X)$. Then the image of the action map:
$$ L(X)\ni \ell \mapsto x_0\cdot \ell \in X$$
will be called a Cartan subtorus $A_X$ of $X$; through the action map, it is isomorphic to a torus quotient of $L(X)$, and the resulting group structure will be considered as part of the data of a Cartan subtorus, as will the resulting identification with the \emph{universal Cartan} (which we will be denoting by the same letter) $A_X\simeq X^\circ /N$ (where $N$ is the unipotent radical of $B$).

In the group case, $X=H$ under the $G=H\times H$ action (defined throughout as a right action, i.e.\ $x\cdot (h_1,h_2) := h_1^{-1} x h_2$), the choice of a Cartan subtorus amounts to a choice of a pair of Borel subgroups $B_1, B_2$ of $H$, together with a point $x_0$ in the corresponding open Bruhat cell. The stabilizer of the point in $B_1\times B_2$ is then a torus, whose centralizer is a Cartan subgroup, whose orbit containing $x_0$ is the ``Cartan subtorus''. 

We will denote by $A$ the universal Cartan of $G$, $A=B/N$, so we have a quotient map of algebraic tori: $A\twoheadrightarrow A_X$, which may not be surjective on $F$-points. The group $L(X)$ constructed above is a Levi subgroup of the parabolic stabilizing the open Borel orbit $X^\circ$,
$$P(X):= \{ g\in G| X^\circ g = X^\circ\}\supset B.$$
This parabolic plays an important role in harmonic analysis, since it gives rise to an ``Arthur $\SL_2$''-parameter -- a deviation from temperedness (when $P(X)\ne B$).  

The vector space $\mathfrak a_X^*:=\Hom(A_X,\Gm)\otimes \QQ$ carries an almost canonical based root system. ``Almost canonical'' refers to the fact that there are different choices in the literature about the length of roots; for a specific choice, we get the root system of the dual group $\check G_X$ of $X$ defined in \cite{SV}. We will return to this root system in section \ref{sec:asymp-basic}, but until then we only need its Weyl group $W_X$ (the ``little Weyl group of $X$'') and its canonical antidominant Weyl chamber $\mathfrak a_X^+\subset\mathfrak a_X$. We have maps:
$$ \mathfrak a_X^+\subset\mathfrak a_X:= \Hom(\Gm,A_X)\otimes\QQ \twoheadleftarrow \Hom(\Gm,B) \otimes\QQ =:\mathfrak a \supset \mathfrak a^+,$$
where $\mathfrak a^+$ denotes the antidominant Weyl chamber corresponding to the universal Cartan of $G$.

The wavefront condition is the condition that the map: $\mathfrak a\to \mathfrak a_X$ sends $\mathfrak a^+$ \emph{onto} $\mathfrak a_X^+$. This technical condition is satisfied for all symmetric varieties and almost all spherical varieties; it is being imposed because this is the case in which the theory of asympotics of \cite[\S 5]{SV} has been completed. (Though, conjecturally, the wavefront condition shouldn't be necessary.)

The other two conditions state that there is a Cartan subtorus $A_X\subset X$, such that the following hold:
\begin{enumerate}
\item The natural map:
\begin{equation}\label{Cartanweak} \tilde \Lambda_X^+:= A_X(F)^+/A(\mathfrak o) \to X(F)/K
\end{equation}
is surjective. Elements of $\tilde\Lambda_X^+$ which map to distinct elements of $\Lambda_X^+ = A_X(F)^+/A_X(\mathfrak o)$ correspond to distinct $K$-orbits on $X$. We also assume that $X$ has a model over $\mathfrak o$ (compatible with that of $G$), and that $X(\mathfrak o)$ consists of the orbits which map to $0\in \Lambda_X^+$ under the map $\tilde\Lambda_X^+\to \Lambda_X^+$.

Here we denote by $A_X(F)^+=A_X^+$ the set of antidominant elements of the torus $A_X$ with respect to the above based root system; that is, the elements of $A_X(F)^+$ are those elements $a$ that satisfy: $|e^\gamma(a)|\ge 1$ for every positive root $\gamma$. (In order to use additive notation on $\mathfrak a_X^*$, we use exponentials to denote actual characters of $A_X$.)

\item There is an Iwahori subgroup $J$ relative to a Borel $B$ used to define $A_X$ such that for every $x\in A_X(F)^+$ we have:
$$ x J = x (J\cap B).$$
\end{enumerate}
\emph{From now on we fix such a Cartan subtorus $A_X$.}

This is the best decomposition that one can hope for in general, and it holds at almost all places if $G$ and $X$ are defined over a global field, as explained in \cite{SaRS}. During the first reading it is advisable to restrict to the case that $A\to A_X$ is surjective on $F$-points, where $A=B/N$. In that case we have: 
$$\tilde \Lambda_X=A_X(F)/A(\mathfrak o) = A_X(F)/A_X(\mathfrak o)=: \Lambda_X,$$ 
which can also be identified with the cocharacter group of $A_X$ via the map: 
$$ \check\lambda \mapsto \check\lambda(\varpi) A_X(\mathfrak o)$$
for any cocharacter $\check\lambda$ into $A_X$. However, such a clean description is in general impossible, as demonstrated by the following example:

\begin{example}\label{exampleSL2}
 Let $X=T\backslash \SL_2$. As a variety, it can also be identified with the quotient of $\PGL_2$ by a torus. Thus, we have a surjection: $X(F)/\SL_2(\mathfrak o)\twoheadrightarrow X(F)/\PGL_2(\mathfrak o)$. One can easily see that $A_X=A_{\PGL_2}$, hence:
 $$ X(F)/\PGL_2(\mathfrak o) \leftrightarrow A_{\PGL_2}(F)^+/A_{\PGL_2}(\mathfrak o) \leftrightarrow \mathbb N.$$
 
 Choose a good, smooth model over $\mathfrak o$ (e.g.: $T$ being the special orthogonal group of an integral, residually nondegenerate quadratic form). Then, under the above parametrization, ``zero'' corresponds to $X(\mathfrak o)$. 
 
 The fibers of the map: 
 $$ \tilde \Lambda_X= A_X(F)^+/A_{\SL_2}(\mathfrak o) \to A_{\PGL_2}(F)/A_{\PGL_2}(\mathfrak o) = \Lambda_X$$ 
 correspond to square classes in $\mathfrak o^\times$. However, it is easy to see that $X(\mathfrak o)$ forms a unique $\SL_2(\mathfrak o)$-orbit. Hence, the map \eqref{Cartanweak} is not injective in that case. 
 
 On the other hand, we claim that for sufficiently large elements of $\tilde\Lambda_X$ the map \emph{is} injective. Indeed, the theory of asymptotics that we will recall below states that on ``very antidominant'' elements of $A_X(F)$ the map \eqref{Cartanweak} has to be injective; more precisely, there is a bijection of ``very large'' elements of $X(F)/K$ and ``very large'' elements of $Y(F)/K$, where $Y$ is the boundary degeneration that we will encounter. In this case, $Y(F)= N^-\backslash\PGL_2$, where $N^-$ is unipotent, and clearly $Y(F)/\SL_2(\mathfrak o) \leftrightarrow A_X(F)/A_{\SL_2}(\mathfrak o)$.
\end{example}

\begin{example}\label{examplegroup1}
 In the group case, $X=H$, $G=H\times H$, we have $\Lambda_X= $ the quotient of $\Lambda_H\times\Lambda_H$ (where $\Lambda_H$ is the coweight lattice of the universal Cartan of $H$) by coweights of the form: $(\check\lambda,-\check\lambda^\vee)$, where for a given coweight $\check\lambda$ of $H$, $\check\lambda^\vee$ denotes the ``dual'' weight, $\check\lambda^\vee = - {^{w_0}\check\lambda}$, $w_0=$ the longest Weyl group element. Thus, $\Lambda_X$ can be identified with $\Lambda_H$, but one needs to specify whether the identification is on the ``left'' or on the ``right'' copy -- the two differ by the operation $\check\lambda\mapsto \check\lambda^\vee$. 
 In either case, the set $\Lambda_X^+$ is the set of antidominant elements of $\Lambda_H$.  We remark that the ``left'' choice gives the \emph{opposite} of the ``obvious'' Cartan decomposition for $H$, i.e.\ an antidominant cocharacter $\check\lambda$ is associated to the coset $K_H \check\lambda(\varpi)^{-1} K_H$, which is the coset of $K_H \check\lambda^\vee(\varpi) K_H$, due to the way that multiplication on the left is defined as a right action.
\end{example}

In sections \ref{sec:Satake}-\ref{sec:validity} we will present a general method for reducing the problem of inverse Satake transforms to a (much easier) problem on horospherical ``boundary degenerations'' of $X$. Then, from section \ref{sec:asymp-basic} on, we will impose additional, strong assumptions on $X$ that allow us to explicitly perform this calculation based on the results of \cite{SaSph}. These additional assumptions \emph{contain} the following:
\begin{enumerate}
\item $\tilde\Lambda_X = \Lambda_X$; in other words, the $F$-points of a Borel subgroup act with a unique open orbit on $X(F)$.
\item $X$ is affine homogeneous, or Whittaker-induced from an affine homogeneous spherical variety of a Levi subgroup in the sense of \cite[\S 2.6]{SV}.
\end{enumerate}
There are \emph{more} assumptions needed, in order to ensure the validity of a theorem of \emph{loc.cit}.\ which we recall as Theorem \ref{SaSphTheorem}; these conditions are of combinatorial nature, can be checked relatively easily in each case, and they are \emph{expected} to be no stronger than the above two; however, I do not know how to prove this. They hold in the group case, of course.

\section{Definition: Satake transform} \label{sec:Satake}

The canonical map of universal Cartans:
$$A\to A_X$$
gives rise to a map with finite kernel between their complex dual tori:
\begin{equation}\label{gp} \check A_X \to \check A.
\end{equation}
Recall that (the complex points of) $\check A_X = \Hom(\Lambda_X, \CC^\times)$ (and similarly for $\check A$); we identify its elements in the standard way with unramified characters of $A_X(F)$, via the identification $A_X(F)/A_X(\mathfrak o)\simeq \Lambda_X$ that we discussed previously.

The map \eqref{gp} is an embedding if and only if $\Lambda_X = \tilde \Lambda_X$, i.e.\ if and only if the map $A\to A_X$ is surjective on $F$-points. In any case, the image of this map will be denoted by $\check A_{X,GN}\subset\check A$; it is the Cartan of the \emph{Gaitsgory-Nadler dual group} of $X$.

The map \eqref{gp} can be used to study the unramified spectrum of $X$, but it requires a correction which takes into account the deviation from temperedness. (For example, for $X=$ a point we have $\check A_X=1$, but the trivial representation does not have trivial Langlands parameter.) For a fixed Borel $B$ we consider
$ \delta_{(X)}^\frac{1}{2} := $ the square root of the modular character (defined as the quotient of right by left Haar measure) of the group $B\cap L(X)$, considered as an unramified character of $B$ and hence as an element of $\check A$. It is stable under the action of $W_X$, and we consider the  $W_X$-equivariant morphism: 
\begin{equation} \label{shift} \check A_X \ni \tilde\chi\mapsto \chi  \delta_{(X)}^\frac{1}{2} \in \check A,
\end{equation}
where $\chi$ is the image of $\tilde \chi$ under $\check A_X\to \check A_{X,GN}$.

\begin{remark}
One can replace every occurrence of $\delta_{(X)}$ in this paper by $\delta_{(X)}^{-1}$ without introducing any errors; indeed, the two elements are conjugate under the Weyl group of $L(X)$, which acts trivially on $\check A_X$, and therefore whether one uses $\delta_{(X)}$ or its inverse plays no role in the restriction of $W$-invariant functions on $\check A$, which is the only setting where this character will appear.
\end{remark}

In order to not get confused between the maps \eqref{gp} and \eqref{shift}, we will be  writing $\delta_{(X)}^\frac{1}{2} \check A_X$ to emphasize that certain restriction maps are taken with respect to \eqref{shift}. When \eqref{gp} is injective, of course, $\delta_{(X)}^\frac{1}{2} \check A_X$ can be identified with the subvariety $\delta_{(X)}^\frac{1}{2}\check A_{X,GN}$ of $\check A$.

Now consider the unramified Hecke algebra $\mathcal H(G,K)$ of $G$, which via the Satake isomorphism is canonically isomorphic to $\CC[\check G]^{\check G} \simeq \CC[\check A]^W$. By restriction to the image of  \eqref{shift} we get a morphism of algebras: 

\begin{equation}\label{restriction} \mathcal H(G,K) \simeq \CC[\check G]^{\check G}  = \CC[\check A]^W \xrightarrow{(*)} \CC[\delta_{(X)}^\frac{1}{2}\check A_{X, GN}]^{W_X} =: \mathcal H_X.
\end{equation}

We set $\mathcal S(X)= C_c^\infty(X)$ and $\Phi^0=$ the characteristic function of $X(\mathfrak o)$ (which, recall, we have assumed to consist of the $K$-orbits that map to $0\in \Lambda_X^+$ under $\tilde\Lambda_X^+\to \Lambda_X^+$). These definitions are the ``correct'' ones only when $X$ is affine, which is the case we will eventually focus on. Then:

\begin{theorem}[{\cite[Theorem 6.2.1]{SaSpc}}]\label{factorsthrough}
 The Hecke algebra $\mathcal H(G,K)$ acts on $\mathcal S(X)^K$ via its quotient (*), and $\mathcal S(X)^K$ is torsion-free as a module for the image of $\mathcal H(G,K)$ under (*).
\end{theorem}

In many cases it is known, and in general it is expected, that the action of $\mathcal H(G,K)$ extends ``naturally'' to an action of $\mathcal H_X$ on $\mathcal S(X)^K$ (and this action is also to be denoted as a convolution: $(h,\Phi)\mapsto h\star\Phi$). When the image of (*) generates $\mathcal H_X$ rationally (i.e.\ generates its field of fractions), such an extension is necessarily unique by the above torsion-freeness statement. Since this covers most of the interesting cases, I will remind of the property characterizing this ``natural'' extension in the general case in the proof of Proposition \ref{equivariance} (see also \cite[Conjecture 6.3]{SaSpc}.

There are several related questions that one might ask in order to enrich the upcoming notion of inverse Satake transforms, for example: whether the action of $\mathcal H_X$ on $\mathcal S(X)^K$ extends further to an action of $\CC[\delta_{(X)}^\frac{1}{2}\check A_X]^{W_X}$. In some cases, the answer is no, at least not in the same ``natural'' way that we alluded to before. The cases where I know that this fails are the cases where the dual group $\check G_X$ cannot be defined (due to ``reflections of type $N$'', s.\ \cite{SV}). 

In any case, in the present paper I ignore such questions. I will restrict to the case when the extension of the action to $\mathcal H_X$ is known, and when it is not known the reader should replace $\mathcal H_X$ in everything that follows by $\mathcal H_X'=$ the image of (*). Until section \ref{sec:inverseSatake}, the exact nature of the extension to $\mathcal H_X$ does not matter for the statements.

We will generally distinguish notationally between an element $h$ of $\mathcal H_X$ considered as an operator on $\mathcal S(X)^K$ (or on $C^\infty(X)^K$), and its ``Satake transform'' $\hat h\in \CC[\delta_{(X)}^\frac{1}{2}\check A_{X,GN}]^{W_X}$. 
For an element $h\in \mathcal H_X$, we will denote by $h^\vee$ the ``dual'' element, characterized by:
$$\widehat{h^\vee}(\delta_{(X)}^\frac{1}{2}\chi):= \hat h(\delta_{(X)}^\frac{1}{2}\chi^{-1})$$
(as polynomials on $\delta_{(X)}^\frac{1}{2}\check A_{X, GN}$).
This is compatible under the above map with the involution on the Hecke algebra $\mathcal H(G,K)$ to be denoted by the same symbol:
$$h^\vee(g):= h(g^{-1}),$$ 
because the latter corresponds to inverting the Satake parameter and, as we noticed in a remark above, $\delta_{(X)}$ is $W$-conjugate to $\delta_{(X)}^{-1}$. 

\begin{definition} \label{defSatake}

The \emph{inverse Satake transform} is the map:
\begin{equation} 
\Sat^{-1}: \CC[\delta_{(X)}^\frac{1}{2}\check A_{X, GN}]^{W_X} \ni \hat h\mapsto h^\vee\star \Phi^0 \in \mathcal S(X)^K 
\end{equation}
The map is injective by the torsion-freeness statement of Theorem \ref{factorsthrough}. The \emph{Satake transform} $\Sat$ is the inverse of this map, defined, of course, only on its image. 

\end{definition}

\begin{remark}
When $\tilde\Lambda_X=\Lambda_X \Leftrightarrow \check A_X = \check A_{X,GN}$, the ring $\mathcal H_X$ can also be identified with the ring of invariant polynomials on the \emph{dual group} $\check G_X$ of $X$ described in \cite{SV}. We will be writing again:
$$ \mathcal H_X \simeq \CC[\delta_{(X)}^\frac{1}{2}\check G_X]^{\check G_X},$$
to remind of the shift when we restrict invariant functions on $\check G$. Notice that by \cite[\S 3.3]{SV} the dual group of $X$ and the element $\delta_{(X)}^\frac{1}{2}$ commute in $\check G$ (where $\delta_{(X)}^\frac{1}{2} = e^{2\rho_{L(X)}} (q^{-\frac{1}{2}})$ in the notation of \emph{loc.\ cit.}), therefore $\check G$-invariants will indeed restrict to $\check G_X$-invariants on $\delta_{(X)}^\frac{1}{2}\check G_X$.
\end{remark}

\begin{example}\label{examplegroup2}  In the group case we have $\check G_X \simeq \check H$, but one must decide whether it is embedded as: $\check H\ni z\mapsto (z,z^c)\in \check H\times \check H$ or $\check H\ni z \mapsto (z^c,z)\in \check H$, where the exponent $~^c$ denotes the Chevalley involution fixing the canonical pinning. This choice has to be done in accordance with the identification $\Lambda_X\simeq \Lambda_H$ as explained in Example \ref{examplegroup1}. In the first case, our Satake transform is \emph{dual} to the usual one (i.e.\ differs by the involution $h\mapsto h^\vee$ on the Hecke algebra), while in the second it is equal to the usual Satake transform.  \end{example}

\section{Boundary degenerations and asymptotics} \label{sec:boundary}

To each $X$ we can associate a horospherical $G$-variety $Y$, denoted $X_\emptyset$ in \cite{SV}, called its (most degenerate) ``boundary degeneration''. We will take it to be homogeneous, in which case it is characterized by the following properties:
\begin{itemize}
 \item $Y$ is homogeneous and horospherical (i.e.\ stabilizers contain maximal unipotent subgroups);
 \item $P(X)=P(Y)$; notice that $P(Y)$ is maximal such that the stabilizer of a point of $Y$ contains the commutator $[P,P]$, where $P$ is a parabolic subgroup \emph{opposite} to $P(Y)$;
 \item $\Lambda_X=\Lambda_Y$ and $\tilde \Lambda_X = \tilde \Lambda_Y$.
\end{itemize}

The Cartan-Iwasawa decomposition for $Y$ states:
\begin{equation}\label{CartanY}
 Y/K \leftrightarrow \tilde\Lambda_Y=\tilde\Lambda_X.
\end{equation}
Evidently, such a bijection can be shifted by any element of the $G$-automorphism group of $Y$ (hence, by any element of $\tilde\Lambda_Y$), but we fix it once and for all in order to state the following theorems; there is a more ``geometric'' realization of $Y$ as an open orbit in a normal bundle, which leads to a rigidification of this decomposition relative to the Cartan decomposition for $X$ (s.\ the proof of Theorem \ref{thmasymp}).

\begin{example}
 In the group case, $X=H$, $G=H\times H$, the boundary degeneration $Y$ is isomorphic to: 
 $$ A^\diag(N\backslash H \times N^-\backslash H),$$
 where $B=AN, B^-=AN^-$ are two opposite Borel subgroups of $H$. (There is, of course, no obvious reason here to present it like that since $B$ and $B^-$ are conjugate; however, this is the presentation that generalizes to the intermediate boundary degenerations, which will not be used in this paper.)
\end{example}

\begin{remark} \label{remarktorusaction}
The ``universal Cartan'' $A_Y=A_X$ of $Y$ acts on $Y$ ``on the left''. We clarify the conventions, which can be a source of confusion. The variety $Y$ is isomorphic to $U^- S \backslash G$, where $U^-$ is the unipotent radical of a parabolic in the class of parabolics \emph{opposite} to $P(X)$ and $S$ is a subgroup of the corresponding Levi $L(X)$ which contains the commutator of the Levi. In this presentation, the universal Cartan of $X$ is: 
$$A_X = L(X)/S \twoheadleftarrow P(X)/U[L(X),L(X)] \twoheadleftarrow B/N = A,$$ where $U$ is the unipotent radical of $P(X)$ and $A$ is the universal Cartan of $G$.

This shows what the natural definition for the action of $A_X$ is, namely, lifting an element $a\in A_X$ to an element $\tilde a \in L(X)$ we have:
$$ a\cdot U^- S x := U^- S \tilde a x \in U^- S \backslash G = Y.$$

For example, if we have a presentation $Y\simeq N\backslash G$ for some maximal unipotent subgroup $N$, we \emph{should not} identify $A_Y=A_X=A$ with the quotient $B/N$, where $B$ is the normalizer of $N$, and let it act in the obvious way via this identification. Instead, if $B$ is our fixed Borel then we should present $Y$ as $N^-\backslash G$ for some unipotent radical $N^-$ of a parabolic $B^-$ \emph{opposite} to $B$, identify $A=B/N$ with the intersection of $B$ and $B^-$, and let it act ``on the left'' as a subgroup of $B^-$. The two actions differ by the action of the longest Weyl element on $A$.

We will return to the $A_X$-action on $Y$ in section \ref{sec:inverseSatake}.
\end{remark}

The basic theorem of asymptotics, restricted to $K$-invariants, is:

\begin{theorem}\label{thmasymp}
 There is a unique $\mathcal H(G,K)$-equivariant morphism:
 $$\Asymp: C^\infty(X)^K\to C^\infty(Y)^K$$
 with the property that, for any $\check\lambda$ ``deep enough'' in $\tilde\Lambda_X^+$, we have:
 $$ \Phi(x_{\check\lambda} K) = \Asymp(y_{\check\lambda} K),$$
 where we denote $\check\lambda \mapsto x_{\check\lambda} K$, resp.\ $y_{\check\lambda} K$, the Cartan decomposition for $X$ (resp.\ $Y$).
\end{theorem}

``Deep enough'' or ``large'' will be used invariantly to signify that the given elemens of a commutative monoid are sufficiently far from its ``walls''.

\begin{proof}
This is \cite[Theorem 5.1.2]{SV}, where this map is denoted by $e_\emptyset^*$, up to showing that the isomorphism: $$\tilde\Lambda_X \simeq \tilde\Lambda_Y$$
can be chosen so that the association induced by the Cartan decomposition:
$$x_{\check\lambda} K  \mapsto y_{\check\lambda} K \,\,\, (\check\lambda\in \tilde\Lambda_X)$$
is compatible with the ``exponential map'' in the sense of \emph{loc.cit.}, \S 4.3. 

Let $\bar X$ be a smooth toroidal embedding of $X$, and let $Z$ be any $G$-orbit in $\bar X$ whose normal bundle contains a subvariety isomorphic to $Y$ (necessarily as its open $G$-orbit). By the local structure theorem of Brion-Luna-Vust (s.\ \emph{loc.cit.} Theorem 2.3.4), there is a $P(X)$-stable open subset $S\subset \bar X$, meeting every $G$-orbit, which is $P(X)$-equivariantly isomorphic to $\overline{A_X} \times U_{P(X)}$, where $\overline{A_X}$ denotes the closure of $A_X$ in $S$. Thus, $\overline{A_X}$ is a smooth toric variety, from which it is easy to see that there is a $P(X)$-equivariant open embedding: 
\begin{equation}\label{emb} A_X\times U_{P(X)} \hookrightarrow N_{S\cap Z} S\end{equation}
(normal bundle to $S\cap Z$ in $S$) and a $p$-adic analytic map:
$$ \xymatrix{
N_{S\cap Z} S \ar[r]^\varphi & S \\
 A_X \times U_{P(X)} \ar@{^{(}->}[u] \ar[r] & A_X\times U_{P(X)} \ar@{^{(}->}[u]
}$$
which is the identity on $S\cap Z$ and on its normal bundle, and the identity on the lower horizontal arrow of the above diagram. 

We can now identify $Y$ with the open $G$-orbit in $N_Z \bar X$ and the subvariety $A_X$ of \eqref{emb} with a Cartan subtorus of $Y$, and under this identification we have:
$$ \varphi(y_{\check\lambda} A(\mathfrak o)) = x_{\check\lambda} A(\mathfrak o),$$
in particular the association:
$$ x_{\check\lambda} K  \mapsto y_{\check\lambda} K$$
is compatible with the exponential map of \emph{loc.cit}.\ \S 4.3.
\end{proof}

It is easily seen from the defining property that $\Asymp$ is dual to a morphism:
\begin{equation}\label{eTheta} 
\Asymp^*:\mathcal M(Y)\to \mathcal M(X),
\end{equation} 
where $\mathcal M(\bullet)$ denotes spaces of \emph{compactly supported smooth measures}, with the property that $1_{y_{\check\lambda} K} \mapsto 1_{x_{\check\lambda} K}$ for large $\check\lambda\in \tilde\Lambda_X^+$ (where $1_S$ denotes the characteristic measure of an open compact subset $S$). 

\begin{remark} In \cite{SV} this map (denoted $e_\emptyset$) was defined between spaces of functions, but here it's more convenient to define it on spaces of measures, thus avoiding some factors in the formulas that follow as well as the need to fix a $G$-eigenmeasure. We point the reader's attention to the fact that $1_{x_{\check\lambda} K}$ etc.\ denote \emph{characteristic measures}, not functions.)
\end{remark}

\section{Range of validity of asymptotics} \label{sec:validity}

We remain, for now, in the general setting where $\tilde \Lambda_X$ is not necessarily equal to $\Lambda_X$; more precisely, $\Lambda_X$ is the quotient of $\tilde\Lambda_X$ by its torsion subgroup. Hence, the complexified dual:
$$ \tilde A_X:= \Hom(\tilde\Lambda_X,\CC^\times)$$
has the natural structure of a complex algebraic group, whose identity component is the torus $\check A_X = \Hom(\Lambda_X,\CC^\times)$. We have natural morphisms:
\begin{equation}
\check A_X \hookrightarrow \tilde A_X \overset{(**)}{\twoheadrightarrow} \check A_{X,GN} \hookrightarrow \check A,
\end{equation}
where the arrow in the middle is obtained by restricting a character to the image of $A(F)$. We let $\tilde\chi\mapsto \chi$ denote the map $(**)$, and we let $\mathfrak f$ denote the kernel of (**); it is the finite group of characters of $A_X(F)$ trivial on the image of $A(F)$.

The following result is proven in \cite{SaSph} under the assumptions of \S \ref{sec:assumptions}. 

\begin{theorem}\label{generaleigenfunctions}
There is a rational family\footnote{A ``rational family'' can be defined as an element of $\Hom (\mathcal M(X)^K, \CC[\tilde A_X]) \otimes_{\CC[\tilde A_X]} \CC(\tilde A_X)$; equivalently, it is a $\CC(\tilde A_X)$-valued function on $X/K$, with only a \emph{finite} number of poles.} of $\mathcal H(G,K)$-eigenfunctions $\tilde A_X \ni \tilde\chi\mapsto \Omega_{\tilde\chi}$ on $X$, with the following properties:
\begin{enumerate}
\item In terms of the Cartan decomposition, $\Omega_{\tilde\chi}$ has the form:
\begin{equation}\label{Omega} \Omega_{\tilde\chi} (x_{\check\lambda}) = q^{\left<\rho_{P(X)},\check\lambda\right>} \sum_{w\in W_X} \sum_{\psi\in \mathfrak f} a_w^\psi(\tilde\chi) (\psi\tilde\chi)(e^{w\check\lambda}),
\end{equation}
for certain rational coefficients $a_w^\psi$, where $\rho_{P(X)}$ is the half-sum of roots in the unipotent radical of $P(X)$. (We use exponential notation when elements of $\tilde\Lambda_X$ are considered as homomorphisms: $\tilde A_X\to \CC^\times$.)

\item  $\mathcal H(G,K)$ acts on $\Omega_{\tilde\chi}$ via the character $\chi\delta_{(X)}^\frac{1}{2}$ (identified with its image in $\check A/W$).

\item The specializations of $\Omega_{\tilde\chi}$ at any Zariski dense subset of $\tilde A_X$ where they are defined span a dense subspace of $\left(\mathcal M(X)^K\right)^*$; in other words, if $\left<\Omega_{\tilde\chi}, \mu\right>=0$ for $\tilde\chi$ in a Zariski dense subset, then $\mu\in \mathcal M(X)^K$ is zero. 
\end{enumerate}

\end{theorem}

\begin{proof}
This is \cite[Theorem 4.2.2]{SaSph} (notice that $\rho_{P(X)}=\rho$ on $\tilde \Lambda_X$), except for the density statement which is \cite[Theorem 6.1.1]{SaSpc}

\end{proof}

\begin{remark} 
 The notation here is slightly different from \emph{loc.cit.}, where $\tilde\chi$ is a character of a certain subgroup $R\subset A(\bar F)$, namely the subgroup of elements which map to $A_X(F)$ under the quotient map: $A\to A_X$. The character $\tilde\chi$ in \emph{loc.\ cit}.\ was varying over all characters of $R$ which extend elements of $\delta_{(X)}^\frac{1}{2} \check A_{X,GN}$ on $A(F)$.  
The above formula is derived from formula (4.2) of \emph{loc.cit}.\ which involves the characters $^w\tilde\chi\delta^{-\frac{1}{2}}$ which \emph{do} descend to characters of $A_X(F)$; more precisely, the character $^w\tilde\chi\delta^{-\frac{1}{2}}$ of \emph{loc.\ cit}.\ is equal to what we presently denote by $\delta_{P(X)}^{-\frac{1}{2}} {^w\tilde\chi}$, which explains the passage from one formula to the other.
\end{remark}

We are ready to draw our first conclusion:

\begin{proposition}\label{exp}
 The morphism $\Asymp^*:\mathcal M(Y)\to \mathcal M(X)$, which a priori maps $1_{y_{\check\lambda} K}$ to $1_{x_{\check\lambda} K}$ only for ``large'' $\check\lambda\in \tilde\Lambda_X^+$, actually has this property for \emph{every} $\check\lambda\in \tilde\Lambda_X^+$.
\end{proposition}

\begin{remark}
 Notice that different $\check\lambda\in \tilde\Lambda_X^+$ with the same image in $\Lambda_X^+$ can correspond to the same $K$-orbit on $X$, as we saw in example \ref{exampleSL2}.
\end{remark}

\begin{proof}
 By the defining property of $\Asymp$, $\Asymp(\Omega_{\tilde\chi})$ has to be an $\mathcal H(G,K)$-eigenfunction on $C^\infty(Y)^K$ with the same eigencharacter, and given by the formula \eqref{Omega} for all large $\check\lambda\in\tilde\Lambda_Y=\tilde\Lambda_X$. The only such eigenfunction is given by the formula \eqref{Omega} for \emph{all} $\check\lambda\in\tilde\Lambda_Y$. 
 
 By the density property, $\Asymp^*(1_{y_{\check\lambda} K})$ is characterized by the property that for (almost) all $\tilde\chi \in \tilde A_X$:
 $$ \left<\Asymp^*(1_{y_{\check\lambda} K}), \Omega_{\tilde\chi}\right> = \left<1_{y_{\check\lambda} K}, \Asymp \Omega_{\tilde\chi}\right>.$$
 
 But this formula holds for $1_{x_{\check\lambda} K}$ in place of $\Asymp^*(1_{y_{\check\lambda} K})$, for $\check\lambda\in \tilde\Lambda_X^+$, by \eqref{Omega}.
\end{proof}

\begin{corollary}\label{corollaryvalidity}
 For any $\Phi\in C^\infty(X)^K$, we have $\Phi = \Asymp (\Phi)|_{\tilde\Lambda_X^+}$ as functions on $\tilde\Lambda_X^+$.
\end{corollary}

This is the key to computing explicitly the inverse Satake transforms of various functions, since it is much easier to compute the Hecke action on $C^\infty(Y)^K$, than on $C^\infty(X)^K$.

 \section{Asymptotics of the basic function} \label{sec:asymp-basic}

 From now on we assume that $X$ is affine homogeneous or Whittaker-induced from an affine homogeneous spherical variety of a Levi subgroup in the sense of \cite[\S 2.6]{SV}. We also require that $\tilde\Lambda_X=\Lambda_X$ (equivalently, $\check A_X =\check A_{X,GN}$, and we will write $\chi$ instead of $\tilde\chi$ for a character of $\Lambda_X$). The formulas that follow will involve the coroot system of $X$ (i.e.\ the root system of its dual group $\check G_X$), as normalized in \cite[\S 3.1]{SV}. The set of positive roots of $\check G_X$ will be denoted by $\check\Phi_X^+$.
 
 In the affine case, the characteristic function of $X(\mathfrak o)$ (which under our present assumptions forms a single $K$-orbit, parametrized by $0\in \Lambda_X^+$) will be denoted by $\Phi^0$.

I point the reader to \cite[\S 2.6]{SV} for the general formalism of Whittaker-induction, but the basic idea is very familiar; in our case, we start with an affine homogeneous variety $H\backslash L$ of a Levi subgroup $L$, and a \emph{generic} character $\Psi: U_P(F)\to \CC^\times$ of the unipotent radical of a parabolic with Levi $L$, such that $\Psi$ is fixed by $H$ (and hence extends to a character of $HU_P$). Then, instead of smooth functions on $X:= HU_P\backslash G$ one considers smooth sections of the induced character (which can be thought of as a complex line bundle $\mathcal L_\Psi$ over the $F$-points of $X$).  Everything that we have established so far extends to the Whittaker-induced case, with the dual group (and hence the set $\Lambda_X^+$ of anti-dominant weights) being \emph{different} from that of $X$ considered as a variety without that line bundle. In this case, the Cartan decomposition does not parametrize all $K$-orbits on $X$ but only the ``relevant'' ones (i.e.\ those which can support $K$-invariant sections of the line bundle). 
Of course, as in the usual case, we need the analogous assumptions of \S \ref{sec:assumptions} to hold for the Cartan decomposition, and they do at almost every place if $X$ is defined over a global field. If, in the presentation above, $HU_P \cdot 1\in X$ is on the orbit represented by $0\in \Lambda_X^+$,  the role of the ``basic function'' here will be played by the section $\Phi^0$ defined by:
 $$ \Phi^0(g) = \begin{cases} \Psi(h), &\,\,\mbox{ if }g=hk, h\in HU_P, k\in K;\\ 0, & \mbox{ otherwise.}\end{cases}$$
 
In either case, from now on we will require that the assumptions of \cite[Theorem 7.2.1]{SaSph} hold; as remarked in \S \ref{sec:assumptions}, this includes, and is expected to be equivalent to, 
the requirement that $X$ is affine homogeneous or Whittaker-induced from such; however, one must for now check additional combinatorial conditions in each case. The case of $X=$ a reductive group satisfies these conditions.

I will not repeat the conditions here (as they involve a long list of definitions), but they have to do with a set $\Theta^+$ of triples $(\check\theta,\sigma_{\check\theta},r_{\check\theta})$, where $\theta^+\in \Lambda_X$, $\sigma_{\check\theta}$ is $+$ or $-1$, and $r_{\check\theta}$ is a half-integer. This set is obtained from the combinatorial invariants of the spherical variety, and in particular the valuations induced by its \emph{colors} ($B$-stable divisors). I refer the reader to \cite[\S 7.1]{SaSph} for the definitions. Roughly speaking, the conditions state that $\Theta^+$ behaves like the set of positive roots of a root system with Weyl group $W_X$: it can be completed to a $W_X$-stable set (where $W_X$ acts on such triples by acting just on $\check\theta$) by inverting the $\check\theta$'s, ``loses'' a prescribed subset of elements when acted upon by a simple reflection etc. We will see some examples below. By abuse of notation, we will sometimes write $\check\theta\in \Theta$, instead of the corresponding triples.

Notice that the condition $\tilde\Lambda_X=\Lambda_X$ implies that $X(\mathfrak o)$ is a single $K$-orbit, $\check A_X = \check A_{X,GN}$, and that the Hecke eigenfunctions of Theorem \ref{generaleigenfunctions} are now parametrized by $\chi\in \check A_X$ (with no finite group $\mathfrak f$ entering in their formula). 

We recall and reformulate the statement of \cite[Theorem 7.2.1]{SaSph} under our present assumptions (more precisely, its restriction to affine or Whittaker-induced from affine cases where, in the notation of the theorem, $\omega=$ a constant):

\begin{theorem} \label{SaSphTheorem}
There is a positive constant $c$ such that the Hecke eigenfunctions $\Omega_\chi$ of Theorem \ref{generaleigenfunctions}, normalized so that their value at $X(\mathfrak o)$ is $1$, are equal to:
$$ \frac{\Omega_\chi(x_{\check\lambda})}{\Omega_\chi(x_0)} = c^{-1} q^{\left<\rho_{P(X)},\check\lambda\right>} \cdot P_{\check\lambda}(\chi),$$
where $P_{\check\lambda}$ is the $W_X$-invariant polynomial on $\check A_X$ given by:

 \begin{equation}\label{sphericalformula} P_{\check\lambda} =  \sum_{w\in W_X} \left(\frac{\prod_{\check\theta\in \Theta^+} (1-\sigma_{\check\theta} q^{-r_{\check\theta}} e^{\check\theta})}{ \prod_{\check\gamma\in\check\Phi_X^+}(1-e^{\check\gamma})} e^{\check\lambda}\right)^w
 \end{equation}

\end{theorem}

\begin{proof}
This is a restatement of \emph{loc.cit}.\ (7.4); see the second formula of the proof for the reformulation that we have presented here. The fact that the $P_{\check\lambda}$s are polynomials can easily be seen from (7.4), where they are expressed in terms of Schur polynomials.
\end{proof}

Moreover: 
\begin{proposition}\label{basis}
The polynomials $P_{\check\lambda}$, for $\check\lambda$ varying over the antidominant elements of $\Lambda_X$, form a basis for the $W_X$-symmetric polynomials on $\check A_X$. 
\end{proposition}

\begin{proof}
This is included in the proof of \cite[Theorem 8.0.2]{SaSph}.
\end{proof}

 \begin{example}\label{examplegroup3}
  In the group case, $X=H$, we have $\Theta^+=\check\Phi_H^+$ (positive coroots of $H$), $\sigma_{\check\theta} = +1$ and $r_{\check\theta} = 1$ for all $\check\theta$, so we get Macdonald's formula according to which:
  $$P_{\check\lambda} = \sum_{w\in W_H} \left(\prod_{\check\gamma\in\check\Phi_H^+} \frac{ 1-q^{-1} e^{\check\gamma}}{ 1-e^{\check\gamma}} e^{\check\lambda}\right)^w.$$
 \end{example}

\begin{example}
In the Whittaker case, $X=N\backslash G$, where $N$ is a maximal uniponent subgroup endowed with a nondegenerate character $\Psi$, we have: $\check G_X = \check G$, $\Theta^+=\emptyset$, and:
$$ P_{\check\lambda} = \sum_{w\in W} \left(\prod_{\check\gamma\in\check\Phi_G^+} \frac{1}{ 1-e^{\check\gamma}} e^{\check\lambda}\right)^w =  \frac{\sum_{w\in W} (-1)^{\ell(w)} e^{\check\rho_B-w \check\rho_B + w\check\lambda}}{\prod_{\check\gamma\in \check\Phi_G^+}(1-e^{\check\gamma})}, $$
where $\ell(w)$ is the length of $w$ and $\check\rho_B = \frac{1}{2} \sum_{\check\gamma\in\check\Phi_G^+} \check\gamma$. The right hand side is, of course, the character (Schur polynomial) of the irreducible representation of $\check G$ with \emph{lowest} weight $\check\lambda$.
\end{example}

\begin{example}
When $X=\Sp_{2n}\backslash \GL_{2n}$ we have $P(X)=$ the standard parabolic with Levi of type $\GL_2 \times \GL_2 \times \cdots \times \GL_2$, and the dual group $\check G_X$ is isomorphic to $\GL_n$ (embedded in $\check G=\GL_{2n}$ via the spherical roots $\alpha_1+2\alpha_2+\alpha_3$, $\alpha_3+2\alpha_4+\alpha_5, \dots$). We have:
$$ P_{\check\lambda} = \sum_{w\in W_X} \left(\prod_{\check\gamma\in \check\Phi_X^+} \frac{1-q^{-2} e^{\check\gamma}}{1- e^{\check\gamma}}\right)^w.$$

It requires a long introduction to the structure of spherical varieties (and the definition of the set $\Theta^+$) in order to explain how these are computed, but I will give a few hints: The calculation of $P(X)$ is easy, and the spherical roots can be read off from the diagrams in Luna's paper \cite{Luna}. One can then compute the $\PGL_2$-spherical varieties corresponding to each simple root $\alpha$ of $G$: these are the varieties $X^\circ P_\alpha/\mathcal R(P_\alpha)$, where $P_\alpha$ is the parabolic whose Levi has a single positive root $\alpha$, $\mathcal R(P_\alpha)$ is its radical, and $X^\circ$ is the open Borel orbit. One sees that for $\alpha_1,\alpha_3, \alpha_5$ etc.\ this $\PGL_2$-variety is a point (which is why $P(X)$ is the standard parabolic containing the negatives of those roots), while for $\alpha_2, \alpha_4, \dots$ they are of the form $N\backslash \PGL_2$, where $N$ is unipotent. This implies that $X^\circ P_{\alpha_{2i}}$ contains a \emph{color}, a $B$-stable divisor, whose valuation gives rise to the element $\check\theta=\check\alpha_{2i}$ of $\Theta^+$. Notice that $\check \alpha_{2i}$ is here equal to the coroot corresponding to the root $\alpha_{2i-1}+2\alpha_{2i}+\alpha_{2i+1}$.
The rest of the triple $(\check\theta,\sigma_{\check\theta}, r_{\check\theta}) = (\check\alpha_{2i},+1, 2)$, and the other elements of $\Theta^+$, can be computed from the definitions of \cite[\S 7.1]{SaSph}.
\end{example}

Before we continue, we need to discuss how we will denote certain functions on the horospherical boundary degeneration $Y$ (and on $X$) as rational functions on $\check A_X/W_X$.

 We introduce a basis of $\mathcal S(Y)^K$ indexed by $\Lambda_X$, where $\check\lambda\in \Lambda_X$ is associated to the function:
\begin{equation}\label{defelambda}
e^{\check\lambda}:= q^{\left<\rho_{P(X)},\check\lambda\right>} \mbox{ times the characteristic function of }y_{\check\lambda} K.
\end{equation}
We will be writing $\hat\Phi$ for the expression of an arbitrary element of $C^\infty(Y)^K$ as a series in the elements $e^{\check\lambda}$, and we will also use rational functions to denote the corresponding power series. Notice that a rational function does not correspond to a unique power series, unless extra conditions are given on the support of the power series, e.g.:
 $$ \frac{1}{1-e^{\check\alpha}}$$
 could correspond to both $\sum_{i\ge 0} e^{i\check\alpha}$ and $-\sum_{i\ge 1} e^{-i\check\alpha}$. 
 
 In what follows, we will fix a strictly convex cone $\mathcal C_X$ in $\Lambda_X$ (i.e.\ the intersection of $\Lambda_X$ with a strictly convex, finitely generated cone in the $\mathbb Q$-vector space it spans) and will require throughout that all our power series have support in a translate of this cone, without the need to repeat this condition every time. (Later, we will also introduce a larger strictly convex cone $\mathcal C_X'$, depending on the function that we want to represent; notice that as long as the latter contains the former and is strictly convex, any rational function that can be expanded as a series in a translate of $\mathcal C_X$, also corresponds unambiguously to the same series if we want to expand it in a translate of $\mathcal C_X'$.) The cone $\mathcal C_X$ is defined as follows: Recall that we assume that $X$ is affine, and we have a decomposition of the coordinate ring:
 \begin{equation}\label{coordring} F[X] = \bigoplus_\chi V_\chi
 \end{equation}
 into a multiplicity-free direct sum of highest weight modules. The set of $B$-weights appearing in this decomposition is actually a saturated monoid of the weights of the quotient torus $A_X$, and we let $\mathcal C_X$ denote the dual cone:
 $$ \mathcal C_X = \{ \check\lambda\in \Lambda_X| \left<\chi,\check\lambda\right> \ge 0 \mbox{ for all $\chi$ appearing in \eqref{coordring}}\}.$$
 
 Since the $\chi$'s appearing in \eqref{coordring} are all dominant, this cone contains the images of all positive coroots of $G$ in $\Lambda_X$.
 
 \begin{example}
  When $G=\SL_2$ and $Y=N\backslash\SL_2$, where $N$ is a maximal unipotent subgroup, and $\check\alpha$ is the positive coroot of its universal Cartan, we have that $\mathcal C_X$ is spanned by the positive coroots (there is no other possibility in this one-dimensional case, no matter what $X$ was) and $P(X)=$ the Borel subgroup. The expression:
  $$ \frac{1}{1-q^{-1} e^{\check\alpha}}$$
  stands for the characteristic function of $\mathfrak o^2\smallsetminus \{0\}$, under the identification of $Y(F)$ with $F^2\smallsetminus\{0\}$. 
  
  Indeed, first of all we expand in a power series in positive powers of $e^{\check\alpha}$, since $\mathcal C_X$ must contain positive multiples of $\check\alpha$. Secondly, we interpret $q^{-i} e^{i\check\alpha} = q^{-i \left<\rho,\check\alpha\right>} e^{i\check\alpha}$ as the characteristic function of the coset $y_{i\check\alpha} K$. Finally, for the Iwasawa decomposition of $Y$ we should fix a Borel $AN^-$ opposite from the ``standard'' one, and use an isomorphism $Y \simeq N^- \backslash G$ to represent $y_{\check\lambda}$ by $\check\lambda(\varpi)\in A(F)$. Then we immediately see that under a suitable embedding of $Y(F)$ in $F^2$ we have: $y_{i\check\alpha} K = $ the subset $(\mathfrak p^i)^2\smallsetminus (\mathfrak p^{i+1})^2$ of $F^2$.
 \end{example}
 
 Now we are ready to describe the image of the basic function under the asymptotics map. Recall that $\Phi^0\in \mathcal S(X)^K$ denotes the characteristic function of $X(\mathfrak o)$. We have the following:
 
\begin{proposition}
  The support of $\Asymp (\Phi^0)$, as a function on $Y/K = \Lambda_X$, belongs to a translate of the cone $\mathcal C_X$.
\end{proposition}

\begin{proof}
 This is \cite[Proposition 5.4.5]{SV}; s.\ also its proof, where the affine embedding containing its support is specified as the horospherical ``affine degeneration'' of $X$, i.e.\ the affine embedding of $Y$ whose coordinate ring, as a $G$-module, contains the same highest weight representations as $F[X]$. 
\end{proof} 
 
This shows that for the calculations that follow we can unambiguously represent functions on $Y/K$ as rational functions, as long as they have a power series expansion supported in a translate of $\mathcal C_X$. In the next section we will do the same with functions on $X/K$, by restricting those power series (functions on $\Lambda_X=\Lambda_Y$) to $\Lambda_X^+$.

Our basic result, now, is the following:
 
 \begin{theorem}\label{theoremBF}
  In the notation above, we have:
  \begin{equation}\label{asympBF}
   \Asymp(\Phi^0) = \frac{ \prod_{\check\gamma\in\check\Phi_X^+}(1-e^{\check\gamma})}{\prod_{\check\theta\in \Theta^+} (1-\sigma_{\check\theta} q^{-r_{\check\theta}} e^{\check\theta})} .
  \end{equation}
 \end{theorem}

 \begin{remark} 
 It follows from the definition of the set $\Theta^+$ in \cite[\S 7.1]{SaSph} that it belongs to the cone $\mathcal C_X$.
\end{remark}

 \begin{proof}
  We begin with an extension of Proposition \ref{exp}:
  
  \begin{proposition}
   For any $\check\lambda\in \Lambda_X$, let:
   $$q^{\left<\rho_{P(X)},\check\lambda\right>} P_{\check\lambda} = \sum_{\check\mu \in \Lambda_X^+} c^{\check\mu}_{\check\lambda} q^{\left<\rho_{P(X)},\check\mu\right>} P_{\check\mu}$$
   be the decomposition into the basis of Proposition \ref{basis}. Then:
   $$ \Asymp^* (1_{y_{\check\lambda}K}) = \sum_{\check\mu \in \Lambda_X^+} c^{\check\mu}_{\check\lambda} 1_{x_{\check\mu}K}.$$
  \end{proposition}
  
  The argument of the proof is an obvious extension of that of Proposition \ref{exp} and will be omitted. Thus, the polynomials $P_{\check\lambda}$, even when $\check\lambda$ is not antidominant, have a meaning of their own! They represent the ``exponential'' map $\Asymp^*$.
  
  Going back to the proof of the theorem, it is now enough to show that the inner product of $\Phi^0$ with $\Asymp^* (1_{y_{\check\lambda}K})$, that is: the coefficient $c_{\check\lambda}^0$, is equal to $q^{\left<\rho_{P(X)},\check\lambda\right>}$ times the coefficient of $e^{\check\lambda}$ in the power series expansion of the right hand side of \eqref{asympBF}. That is, we need to show that the coefficient of $e^{\check\lambda}$ is equal to the constant $c^0_{\check\lambda}$ in the notation of the last proposition.
  
  It is shown in \cite[\S 9]{SaSph} that the polynomials $P_{\check\lambda}$, with $\check\lambda$ antidominant, are orthogonal with respect to the inner product:
  $$ [P,Q] =  \int_{\check A_X^1/W_X} P(\chi)\cdot \overline{Q(\chi)} \cdot \frac{ \prod_{\check\gamma\in\check\Phi_X}(1-e^{\check\gamma})}{\prod_{\check\theta\in \Theta} (1-\sigma_{\check\theta} q^{-r_{\check\theta}} e^{\check\theta})}(\chi) d\chi,$$
  where $\check A_X^1$ denotes the maximal compact subgroup of $\check A_X$.
  
  In particular, since $P_0$ is equal to the (positive) constant $c$ of Theorem \ref{SaSphTheorem}, for arbitrary $\check\lambda\in \Lambda_H$ we have:
  $$ [P_{\check\lambda}, P_0] = c \cdot \int_{\check A_X^1/W_X} \sum_{w\in W_X} \left(\frac{\prod_{\check\theta\in \Theta^+} (1-\sigma_{\check\theta} q^{-r_{\check\theta}} e^{\check\theta})}{ \prod_{\check\gamma\in\check\Phi_X^+}(1-e^{\check\gamma})} e^{\check\lambda}\right)^w \cdot \frac{ \prod_{\check\gamma\in\check\Phi_X}(1-e^{\check\gamma})}{\prod_{\check\theta\in \Theta} (1-\sigma_{\check\theta} q^{-r_{\check\theta}} e^{\check\theta})}(\chi) d\chi$$
  $$ = c\cdot \int_{\check A_X^1} \frac{ \prod_{\check\gamma\in\check\Phi_X^+}(1-e^{-\check\gamma})}{\prod_{\check\theta\in \Theta^+} (1-\sigma_{\check\theta} q^{-r_{\check\theta}} e^{-\check\theta})} e^{\check\lambda} (\chi) d\chi. $$
  
  By a complex analysis/contour shift argument one can see that for the probability measure on $\check A_X^1$ this integral is equal to the constant term of the power series that one gets by expanding the inverse of the denominator in the obvious way: $(1-\sigma_{\check\theta} q^{-r_{\check\theta}} e^{-\check\theta})^{-1} =\sum_{i\ge 0}  (\sigma_{\check\theta} q^{-r_{\check\theta}} e^{-\check\theta})^i.$
  
  Thus, the result of the calculation is $c$ times the coefficient of $e^{-\check\lambda}$ in the power series expansion of $\frac{ \prod_{\check\gamma\in\check\Phi_X^+}(1-e^{-\check\gamma})}{\prod_{\check\theta\in \Theta^+} (1-\sigma_{\check\theta} q^{-r_{\check\theta}} e^{-\check\theta})}$, or equivalently $c$ times the coefficient of $e^{\check\lambda}$ on the right hand side of \eqref{asympBF}.
  
  On the other hand, this argument shows that $ [P_0,P_0] = c$, and hence:
  $$c_{\check\lambda}^0 = \frac{[P_{\check\lambda},P_0]}{[P_0,P_0]} = \mbox{ the coefficient of $e^{\check\lambda}$ on the right hand side of \eqref{asympBF}.}$$

 \end{proof}

\begin{example}\label{examplegroup4}
 In the group case $X=H$ we have:
 $$ \Asymp(\Phi^0) = \prod_{\check\gamma\in \check\Phi_H^+} \frac{1-e^{\check\gamma}}{1-q^{-1}e^{\check\gamma}}$$
(cf.\ Example \ref{examplegroup3}). 
\end{example}

\section{Inverse Satake transforms} \label{sec:inverseSatake}

Now recall that ``Hecke'' ring $\mathcal H_X$ acting on $\mathcal S(X)^K$. Recall that we distinguish notationally between an element $h$ of this ring considered as an operator on $\mathcal S(X)^K$ (or on $C^\infty(X)^K$), and its ``Satake transform'' $\hat h\in \CC[\delta_{(X)}^\frac{1}{2}\check A_X]^{W_X}$. (Under our present assumptions: $\check A_X = \check A_{X,GN}$.)

We also have the torus $A_X$ acting ``on the left'' on $Y$. The reader should necessarily read Remark \ref{remarktorusaction}, to avoid potential confusion about the $A_X$-action on $Y$ as we discuss the Satake isomorphism.

Accordingly, the torus $A_X$ acts on $C^\infty(Y)^K$; we normalize this action as:
$$ a \cdot f (y):= \delta_{P(X)}^{\frac{1}{2}}(a) f(a y),$$
so that it is unitary on the subspace of $L^2$-functions, and
define the action of its Hecke algebra $\mathcal H(A_X, A_X(\mathfrak o))\simeq \CC[\Lambda_X] \simeq  \CC[\delta_{(X)}^\frac{1}{2}\check A_X]$  on $A_X(\mathfrak o)$-invariant functions accordingly:
$$ h\star f(y) := \int_{A_X/A_X(\mathfrak o)\simeq \Lambda_X}  a\cdot f(y) h(y).$$

Notice the isomorphism $\CC[\Lambda_X] \simeq  \CC[\delta_{(X)}^\frac{1}{2}\check A_X]$, which is simply coming from the canonical isomorphism $\CC[\Lambda_X]=\CC[\check A_X]$ composed with the obvious identification (translation by $\delta_{(X)}^\frac{1}{2}$) between $\check A_X$ and $\delta_{(X)}^\frac{1}{2}\check A_X$. We will insist on introducing this shift, as we did in \eqref{restriction}, for compatibility reasons with the Satake isomorphism that we are about to discuss. Despite the fact that these isomorphisms are completely canonical, for an element $h\in \mathcal H(A_X, A_X(\mathfrak o))\simeq \CC[\Lambda_X]$ we will write $\hat h$ for its image in $\CC[\delta_{(X)}^\frac{1}{2}\check A_X]$.

With this definition of the action, the characteristic measure of $\check\lambda\in \Lambda_X$ takes the function that we denoted above by $e^{\check\mu}$ to  $e^{\check\mu-\check\lambda}$; this explains our choice of basis. We can formulate this in terms of $h$ and $\hat h$:

\begin{lemma}\label{action}
For $\Phi\in \mathcal S(Y)^K$, let $\hat \Phi$ be its expression in the basis $(e^{\check\lambda})_{\check\lambda\in \Lambda_X}$, thought of as an element of $\CC[\check A_X]$.

Let $h\in \mathcal H(A_X, A_X(\mathfrak o))$, and denote as before by $h^\vee$ the dual element: $h^\vee(a)=h(a^{-1})$. Then:

$$ \widehat{h^\vee\star \Phi} = \hat h \cdot \hat \Phi.$$
\end{lemma}

The definition of this action is compatible with the action of $\mathcal H(G,K)$ under the usual Satake isomorphism:
$$\mathcal S: \mathcal H(G,K) \xrightarrow\sim \CC[\check A]^W,$$
in the following sense:

\begin{lemma}\label{compatible}
Let $h\in \mathcal H(G,K)$, $h' \in \CC[\Lambda_X] = \CC[\delta_{(X)}^\frac{1}{2} \check A_X]$ such that the Satake transform $\hat h$ of $h$, when  composed with the restriction map:
$$ \mathcal H(G,K)\simeq \CC[\check G]^{\check G} \xrightarrow{(*)} \CC[\delta_{(X)}^\frac{1}{2}\check A_X],$$
(where $(*)$ is as in \eqref{restriction}) is equal to $\widehat{h'} \in \CC[\delta_{(X)}^\frac{1}{2} \check A_X]$.

Then, for all $\Phi\in \mathcal S(Y)^K$, we have:
$$ h\star \Phi = h'\star \Phi,$$
where, obviously, the convolution on the left is with respect to the $G$-action and the convolution on the right is with respect to the $A_X$-action.
\end{lemma}

\begin{proof}
It is enough to show this for the special case $Y=N^-\backslash G$. For, any other $Y'$ is a quotient of this of the form $S\backslash G$, where $S$ lives between a parabolic $P(Y')^-$ containing $N^-$ and its derived group, and the action of $\CC[\check A]$ on $\mathcal S(N^-\backslash G)$ descends to $\mathcal S(S\backslash G)$ via the corresponding restriction map:
$$\CC[\check A]\to \CC[\delta_{(Y')}^\frac{1}{2} \check A_{Y'}],$$
where $\delta_{(Y')}$ is defined in complete analogy with $\delta_{(X)}$ earlier. 

When $Y=N^-\backslash G = N\backslash G$, this is the setting of the original Satake transform. Following \cite[(3.4)]{Gross-Satake}, the  Satake transform of an element $h\in \mathcal H(G,K)$ considered as a function (fixing the Haar measure $dg$ on $G$ which gives mass $1$ to $K$), is defined as the following function on the universal Cartan $A$:
$$\mathcal Sh(a) := \delta_B(a)^\frac{1}{2} \int_N h(an) dn,\,\, (a\in A)$$
where $B$ is any Borel subgroup, $A$ is identified with its reductive quotient, and the measure on $N$ is such that $dg = \delta_B(t)^{-1} dn dt dk$ according to the Iwasawa decomposition $G=NTK$, where $T$ is a Cartan subgroup of $B$ and $dt (T(\mathfrak o))=1$.

Let us say that $B$ is chosen opposite to the subgroup $N^-$ above, and let $T=B\cap B^-$, where $B^-$ is the normalizer of $N^-$. The embedding $T\hookrightarrow B\twoheadrightarrow A$ identifies $T$ with $A$. Let $w\in K$ be an element which belongs to the normalizer of $T$ and corresponds to the longest element of the Weyl group, then \emph{by the invariance of $\mathcal Sh(t)$ under $W$} we have, for $t\in T\simeq A$:
$$\mathcal Sh(t) = \mathcal Sh(wtw^{-1}) = \delta_B(t)^{-\frac{1}{2}} \int_N h(wtw^{-1}n) dn = \delta_B(t)^{-\frac{1}{2}} \int_{N^-} h(tn) dn =$$ $$ = \delta_B(t)^{-\frac{1}{2}} h\star \Phi^0 (N^- t^{-1}).$$

Hence, for $\check\mu\in \Lambda_Y$, the evaluation of $\mathcal Sh$ at the associated representative $y_{\check\mu}\in A$ is:

\begin{tabular}{rl}
$\mathcal Sh(y_{\check\mu}) =$ & the coefficient of $e^{-\check\mu}$ when we write $h\star \Phi^0$ in the basis \\ & consisting of the elements $e^{\check\lambda}$ that we introduced above.
\end{tabular}

It follows from Lemma \ref{action} that if $h'\in \mathcal H(A,A(\mathfrak o))$ is the element $\mathcal Sh(t)dt$ (so that $\hat h = \widehat{h'} \in \CC[\check A]$ as required by the present lemma), then the coefficient of $e^{-\check\mu}$ in $\widehat{h'\star \Phi^0}$ is equal to its coefficient in $\widehat{h'^\vee}$, i.e.\ equal to the coefficient of $e^{\check\mu}$ in $\widehat{h'}$, i.e.\ equal to  $\mathcal Sh(y_{\check\mu})$. Hence:
$$ \widehat{h'\star \Phi^0} = \widehat{h\star \Phi^0}.$$

The same has to hold if we replace $\Phi^0$ by any element of $\mathcal S(Y)^K$, since it generates all of them under the action of $\mathcal H(A,A(\mathfrak o))$, and the actions of $A$ and $G$ commute.
\end{proof}

Now we come to combining the theory of asymptotics with the explicit formulas of the previous section. Notice that under the assumptions of the present section we have:
\begin{proposition}
The whole ring $\mathcal H_X = \CC[\delta_{(X)}\check A_X]^{W_X}$ acts on $\mathcal S(X)^K$, and the Satake transform of Definition 3.2 is an isomorphism:
$$ \Sat:  \mathcal S(X)^K \xrightarrow\sim \mathcal H_X.$$
\end{proposition}

\begin{proof}
This is \cite[Theorem 8.0.2]{SaSph}.
\end{proof}

On the other hand, we may identify $\mathcal H_X$ as the subring of $W_X$-invariants in the ring $\CC[\delta_{(X)}\check A_X]$ acting on $\mathcal S(Y)^K$. Then:
\begin{proposition}\label{equivariance}
The asymptotics map $\Asymp: C^\infty(X)^K\to C^\infty(Y)^K$ is $\mathcal H_X$-equivariant. 
\end{proposition}

\begin{proof}
This is easy to see if the image of the restriction map:
$$\CC[\check A]^W \to \CC[\delta_{(X)}^\frac{1}{2}\check A_X]^{W_X}\simeq \mathcal H_X$$
generates $\mathcal H_X$ rationally (i.e.\ generates its field of fractions), or equivalently: when the map $\check A_X/W_X\to \check A/W$ is generically injective.

This is the case in most examples of wavefront spherical varieties, and certainly in the case of symmetric varieties (see the -- stronger -- notion of ``generic injectivity'' in \cite[\S 14.2]{SV}, and Lemma 15 in Delorme \cite{Delorme-Plancherel} which proves it in the symmetric case). 

Indeed, since $\Asymp$ is equivariant under the action of $\mathcal H(G,K)\simeq \CC[\check A]^W$, and by Lemma \ref{compatible} the action of the latter is compatible with the action of $\CC[\delta_{(X)}^\frac{1}{2}\check A_X]$ ``on the left'', it follows that the asymptotics map is equivariant with respect to the image of the restriction map, considered as a subring of $\mathcal H_X$. Since the modules are torsion-free, if the image generates $\mathcal H_X$ rationally then it has to be equivariant with respect to the whole ring $\mathcal H_X$. 

I sketch the proof in the general case: The action of $\mathcal H_X$ on $\mathcal S(X)^K$ was characterized in \cite{SaSph} by the requirement of being equivariant with respect to certain operators which in the literature (although not in \emph{loc.cit}.) are sometimes called ``Eisenstein integrals''. These are a certain rational family of operators:
$$S_\chi: \mathcal S(X) \to I_{P(X)}(\chi),$$
where $I_{P(X)}(\chi)$ denotes the normalized principal series induced from the character $\chi$ of $P(X)$, 
as $\chi$ varies in $\check A_X$. The operators (or rather, the functional obtained by composing with ``evaluation at $1$'') are defined in some convergent region by an integral on the open $P(X)$-orbit, and extended rationally to all $\check A_X$. Again, our parametrization of characters is shifted by $\delta_{(X)}^\frac{1}{2}$ compared to that of \cite{SaSph}, cf.\ the remark following Theorem \ref{generaleigenfunctions}. The action of the ring $\mathcal H_X$ on $I_{P(X)}(\chi)$ is defined to be by the scalar obtained by evaluation at $\delta_{(X)}^\frac{1}{2}\chi$.

The same operators can be defined on $\mathcal S(Y)$, and \cite[Proposition 5.4.6]{SV} states that the following diagram commutes:

$$\xymatrix{
 \mathcal S(X) \ar[dd]_{\Asymp}\ar[dr]  \\
& I_{P(X)}(\chi) \\
\mathcal S_X(Y)  \ar[ur],\\
}
$$
where by $\mathcal S_X(Y)$ we denote the image of $\mathcal S(X)$ in $C^\infty(Y)$.

This suffices to show equivariance of the asymptotics map under $\mathcal H_X$.
\end{proof}

In particular:

\begin{proposition}\label{propinversion}
Let $h\in \mathcal H_X$ and $\Phi\in C^\infty(X)^K$. Then: 
$$\widehat{\Asymp(h^\vee\star \Phi)} = \hat h \cdot \Asymp(\Phi),$$
both sides thought of as formal series in $\Lambda_X$ (supported on a translate of the cone $\mathcal C_X$).

In particular, 
$$ \Asymp(\Sat^{-1}(\hat h)) = \hat h \cdot \Asymp(\Phi^0).$$

Moreover, considered as functions on $\Lambda_X^+$, $\Sat^{-1}(\hat h)$ and $\hat h \cdot \Asymp(\Phi^0)$ coincide.
\end{proposition}

\begin{proof}
This is a combination of what has been proven thus far, namely: Proposition \ref{equivariance}, Lemma \ref{action} (which is easily extended to functions represented by power series) and Corollary \ref{corollaryvalidity}.
\end{proof}

We would like to extend the above to suitable series of elements in $\mathcal H_X$, when $\Phi\in \mathcal S(X)^K$ (so that $\Asymp(\Phi)$ is supported on a translate of the cone $\mathcal C_X$). More precisely, let $\mathcal K_X$ denote all series in the elements $e^{\check\lambda}$, $\check\lambda\in \Lambda_X$ which are $W_X$-invariant, and such that \emph{their support on every translate of the sublattice spanned by $\mathcal C_X$ is compact}. Via the identification $\mathcal H_X = \CC[\Lambda_X]^{W_X}$, the ring $\mathcal H_X$ is the subring of compactly supported elements of $\mathcal K_X$.

\begin{remark}
The assumption on support places a strong restriction on $X$, if $\mathcal K_X$ is to contain elements of non-compact support. Namely, $\mathcal C_X$ should not (rationally) span the whole lattice $\Lambda_X$, which implies that there is a non-trivial eigenfunction of $G$ on the coordinate ring $F[X]$. This in essence leaves out varieties which, in the language of \cite{SV}, are not ``factorizable'', which should eventually be included in the theory.

\end{remark}

Let $\Lambda'_X = $the quotient of $\Lambda_X$ by the sublattice rationally spanned by $\mathcal C_X$, and let $\det$ denote the quotient map. We will use the same symbol to denote the map $X\to \Lambda'_X$ induced by the Cartan decomposition $X/K=\Lambda_X^+$. For $\hat h\in \mathcal K_X$ we have a decomposition:
$$ \hat h = \sum_{\check\delta\in\Lambda'_X} \hat h_\delta,$$
where $\hat h_\delta \in \CC[\Lambda_X]^{W_X}\simeq \mathcal H_X$ is supported in $\det^{-1}(\delta)$, and if $h_\delta$ is the operator on $\mathcal S(X)^K$ corresponding to $\hat h_\delta$ then, for any $\Phi\in\mathcal S(X\cap \det^{-1}(\delta_1))^K$ we have $h_\delta\star \Phi \in \mathcal S(X \cap \det^{-1}(\delta_1+\delta))^K$. Hence, for every $\Phi\in \mathcal S(X)$ we have a well-defined element:
$$ h \star \Phi\in C^\infty(X)^K$$
whose support on every subset of the form $\det^{-1}(\delta_1)$ is compact.

In particular, the inverse Satake transform extends to $\mathcal K_X$:
$$ \Sat^{-1}(\hat h):= h^\vee \star \Phi^0.$$

Now let $\check\rho \in \Lambda_X^+$ be outside of the rational span of $\mathcal C_X$; in particular, the cone $\mathcal C_X'$ spanned by $\check\rho$ and $\mathcal C_X$ is strictly convex. We write $\hat h = L(\check\rho)\in \CC(\delta_{(X)}^\frac{1}{2}\check A_X)^{W_X}$ for the rational function:
$$ \delta_{(X)}^\frac{1}{2} \chi\mapsto  \det\left(I-\chi|_{V_{\check\rho}}\right)^{-1},$$
where $V_{\check\rho}$ is the irreducible module of $\check G_X$ of \emph{lowest weight} $\check\rho$. We consider it as a power series with support in the cone $\mathcal C_X'$; t satisfies the assumptions of the preceding discussion, i.e.\ it belongs to $\mathcal K_X$. 

A combination of Proposition \ref{propinversion} and Theorem \ref{theoremBF} now gives:

\begin{theorem}\label{maintheorem}
Let $\hat h = L(\check\rho)\in \CC(\delta_{(X)}^\frac{1}{2}\check A_X)^{W_X}$, then, as functions on $X/K=\Lambda_X^+$:
 \begin{equation}
\Sat^{-1}\left( L(\check\rho) \right)= h^\vee\star \Phi^0 = \left. L(\check\rho) \cdot \frac{ \prod_{\check\gamma\in\check\Phi_X^+}(1-e^{\check\gamma})}{\prod_{\check\theta\in \Theta^+} (1-\sigma_{\check\theta} q^{-r_{\check\theta}} e^{\check\theta})} \right|_{\Lambda_X^+}.
 \end{equation}
We remind that the term $e^{\check\lambda}$ on the right hand side should be interpreted as $q^{\left<\rho_{P(X)},\check\lambda\right>}$ times the characteristic function of $x_{\check\lambda} K$.
\end{theorem}

\section{The group case; relation to the formula of Wen-Wei Li} \label{sec:groupcase}

We now restrict to the case $X=H$, $G=H\times H$, where $H$ is a (split) reductive group over $F$. Here we have $\check G_X = \check H$ but, according to Example \ref{examplegroup2}, there is a choice to be made in the identification. We make the choice to embed $\check H$ in $\check G$ as:
$$ z\mapsto (z^c, z)$$
in order for our Satake transform to be compatible with the usual one. Recall from Example \ref{examplegroup1} that this choice is compatible with the ``obvious'' Cartan decomposition in terms of antidominant coweights, i.e.\ $\check\lambda\in \Lambda_H^+$ corresponds to the double $K_H$-coset of $\check\lambda (\varpi)$, where $\varpi$ is a uniformizer of our field.

In the group case we have $\mathcal C_X=$ the cone spanned by the positive coroots. Assume that $\check\rho\in \Lambda_H^+$ is outside of the rational span of positive coroots, and let $L(\check\rho)$, as above, be the rational function:
$$ \det(I-\bullet|_{V_{\check\rho}})^{-1},$$
where $V_{\check\rho}$ is the irreducible representation of $\check H$ of lowest weight $\check\rho$. Then we have:

\begin{corollary}\label{corollarygroup}
Let $h$ be the series of elements in the Hecke algebra whose (usual) Satake transform is $\hat h = L(\check\rho)$. Identify it with a function on $K_H\backslash H/K_H$ by fixing the Haar measure on $H$ which is $1$ on $K_H$. 

Then the value of $h$ on $\check\lambda(\varpi)$, $\check\lambda\in \Lambda_H^+$, is equal to $q^{\left<\rho_{B_H},\check\lambda\right>}$ (where $\rho_{B_H}$ is the half-sum of positive roots on $H$) times the coefficient of $e^{\check\lambda}$ in the power series of:
\begin{equation}\label{eqcorollarygroup}
 L(\check\rho) \cdot \prod_{\check\gamma\in\check\Phi_H^+} \frac{1-e^{\check\gamma}}{1-q^{-1} e^{\check\gamma}} .
\end{equation}
(expanded in terms of the cone spanned by $\check\rho$ and the positive coroots).
\end{corollary}

\begin{proof}
Immediate combination of Theorem \ref{maintheorem} and Example \ref{examplegroup3}, using the fact that in this case, for $\check\lambda\in \Lambda_X = \Lambda_H$, $\left<\rho_B,\check\lambda\right> = \left<\rho_{B_H},\check\lambda\right>$.
\end{proof}

\begin{example}[Godement-Jacquet for $\GL_2$]
 We choose a basis $\epsilon_1,\epsilon_2$ for the cocharacter lattice of $H=\GL_2$ such that the positive coroot is $\check\alpha=\epsilon_1-\epsilon_2$. We set $z_i = e^{\epsilon_i}$. The lowest weight of the standard representation of $\check H$ is $\epsilon_2$, and hence:
 $$\hat h:= L(\std) = \frac{1}{(1-z_1)(1-z_2)}.$$
 
 Thus:
 \begin{equation}\label{GJ2} h^\vee\star \Phi^0 = \left. \frac{1-\frac{z_1}{z_2}}{1-q^{-1}\frac{z_1}{z_2}} \cdot \frac{1}{(1-z_1)(1-z_2)}\right|_{\Lambda_H^+}.
 \end{equation}

 It is not immediately evident that this is related to the characteristic function of $\Mat_2(\mathfrak o)$. However, notice that the above expression is equal to:
 $$ \frac{1}{(1-q^{-1}\frac{z_1}{z_2})(1-z_2)} - \frac{\frac{z_1}{z_2}}{(1-q^{-1}\frac{z_1}{z_2})(1-z_1)}.$$
 
 The second summand can be discarded, as its support does not meet the set of antidominant coweights. (The set of antidominant coweights corresponds to monomials of the form $z_1^i z_2^j$ with $j\ge i$.)
 
 The support of the first summand intersects the antidominant coweights on the set of monomials $z_1^i z_2^j$ with $j\ge i\ge 0$, and the coefficient of such a monomial is: $q^{-i}$.
 
 On the other hand, $\left<\rho,i\epsilon_1 +j \epsilon_2\right> = \frac{i-j}{2}$. Therefore, the characteristic function of the coset of the element $\left(\begin{array}{cc} \varpi^i \\ & \varpi^j\end{array}\right)$ appears with coefficient:
 $$ q^{-i+\frac{i-j}{2}} = q^{-\frac{i+j}{2}} = |\det|^\frac{1}{2},$$ and the function of $\eqref{GJ2}$ is equal to $|\det|^\frac{1}{2}$ times the characteristic function of $\Mat_2(\mathfrak o)$. 
 
 (We remind that for $\GL_n$ and $\hat h= L(\std)$ one has: $h^\vee\star \Phi^0= |\det|^{\frac{n-1}{2}}$ times the characteristic function of $\Mat_n(\mathfrak o)$.)

\end{example}

Finally, let us discuss the relationship of Corollary \ref{corollarygroup} with the result \cite[Theorem 5.3.1]{WWL} of Wen-Wei Li. I am very grateful to Wen-Wei Li for explaining this relationship to me and allowing me to reproduce the arguments here.

As in the setting of Corollary \ref{corollarygroup}, we fix an anti-dominant coweight $\check\rho$ which is not contained in the linear span of coroots. We assume that the character group of $H$ is generated by a character ``$\det$'' with the property that $\left<\det,\check\rho\right>=1$. (This is not a serious restriction, and one can recover Corollary \ref{corollarygroup} directly from the formulation of Wen-Wei Li, without this assumption.)

To introduce the result of \cite{WWL} we let $\Psi$ denote the multiset arising as the \emph{multiset union} of the set $\check\Phi_H^+$ of positive corrots of $H$ and the multiset $V$ of weights of the representation $\check\rho$ of the dual group. Consider the product:
$$\prod_{\check\nu\in \Psi}\frac{1}{1-qe^{-\check\nu}},$$
thought of as a series in $\span_{\mathbb Z_{\ge 0}}(\Psi)$ with coefficients in $\mathbb Z[q]$. For $\check\mu\in \Lambda_H$, let $\mathscr P_\Psi(\check\mu,q)\in\mathbb Z[q]$ denote the coefficient of $e^{\check\mu}$ in the above power series, i.e.\ formally:
$$ \sum_{\check\mu\in \Lambda_H} \mathscr P_\Psi(\check\mu,q) e^{\check\mu} = \prod_{\check\nu\in \Psi}\frac{1}{1-qe^{-\check\nu}}.$$

The formula of \cite[Theorem 5.3.1]{WWL} asserts that $\Sat^{-1}(L(\check\rho))$ equals the restriction to $\Lambda_H^+$ of the following series in our basis elements $e^{\check\mu}$:
\begin{equation}\label{eqWWL} \sum_{\check\mu\in \Lambda_H, \left<\det,\check\mu\right>\ge 0} c_{\check\mu}(q)e^{\check\mu},
\end{equation}
with:
$$ c_{\check\mu} (q) = \sum_{w\in W_H} (-1)^{\ell(w)} \mathscr P_\Psi(\check\rho_B- w\check\rho_B, -\check\mu; q^{-1}) q^{\left<\det,\check\mu\right>},$$
where $\check\rho_B$ denotes half the sum of positive coroots of $\check H$ and $\ell(w)$ is the length of $w$.

The coefficient $c_{\check\mu}(q)$ was defined only for $\check\mu\in \Lambda_H^+$ in \cite{WWL}, but the same formula works in general.

To establish the equivalence between the two formulas, we first observe that the coefficients $\mathscr P_\Psi(\check\mu,q)$ admit an alternative presentation where the roles of $\check\Phi_H^+$ and $V$ in $\Psi$ are distinguished, namely:
$$ \sum_{\check\mu\in \Lambda_H} \mathscr P_\Psi(-\check\mu,q^{-1}) q^{\left<\det,\check\mu\right>} e^{\check\mu} = \prod_{\check\gamma\in\check\Phi_H^+} \frac{1}{1-q^{-1}e^{\check\gamma}}\prod_{\check\nu \in V}\frac{1}{1-e^{\check\nu}}.$$
Indeed, for a certain $\mathbb Z_{\ge 0}$-combination of elements of $\Psi$ to be equal to $\check\mu$, we must have $k := \left<\det,\check\mu\right>\ge 0$,
 and the elements from $V$ must be used exactly $k$ times. 
 
Combining this with the definition of $c_{\check\mu}$, changing $\check\mu$ to $\check\mu+\check\rho_B-w\check\rho_B$ and taking into account that 
$\left< \det, \check\mu+\check\rho_B-w\check\rho_B\right> = \left<\det,\check\mu\right>,$
the series \eqref{eqWWL} can be written as:
$$\sum_{\check\mu\in \Lambda_H, \left<\det,\check\mu\right>\ge 0} \mathscr P_\Psi(-\check\mu;q^{-1}) q^{\left<\det, \check\mu\right>}\sum_{w\in W_H} (-1)^{\ell(w)} e^{\check\mu +\check\rho_B-w\check\rho_B} = $$
$$ = \prod_{\check\gamma\in\check\Phi_H^+} \frac{1}{1-q^{-1}e^{\check\gamma}}\prod_{\check\nu \in V}\frac{1}{1-e^{\check\nu}} \sum_{w\in W_H} (-1)^{\ell(w)} e^{\check\rho_B-w\check\rho_B}.$$

Finally, invoking the Weyl denominator formula:
$$\prod_{\check\gamma\in \check\Phi_H^+} (1 - e^{\check\gamma}) = \sum_{w\in W_H} (-1)^{\ell(w)} e^{\check\rho_B-w\check\rho_B}$$
we get:
$$ \prod_{\check\gamma\in\check\Phi_H^+} \frac{1-e^{\check\gamma}}{1-q^{-1}e^{\check\gamma}}\prod_{\check\nu \in V}\frac{1}{1-e^{\check\nu}} = L(\check\rho)\cdot \prod_{\check\gamma\in\check\Phi_H^+} \frac{1-e^{\check\gamma}}{1-q^{-1}e^{\check\gamma}},$$
which is equal to \eqref{eqcorollarygroup}.

\bibliographystyle{alphaurl}
\bibliography{inverseSatake}

\end{document}